\documentclass[11pt, a4paper]{amsart}
\usepackage{amscd,amssymb,eucal,eufrak,mathrsfs,amsmath}
\usepackage{hyperref} 

\input xy
\xyoption{all}
\usepackage{epsfig}

\newtheorem{thm}{Theorem}[section]
\newtheorem{cor}[thm]{Corollary}
\newtheorem{lem}[thm]{Lemma}
\newtheorem{prop}[thm]{Proposition}

\newtheorem{question}[thm]{Question}     
  \newtheorem{f}[thm]{Fact}                              
\theoremstyle{definition}
\newtheorem{defin}[thm]{Definition}

\theoremstyle{remark}
\newtheorem{remark}[thm]{Remark}

\newtheorem{example}[thm]{Example}

\numberwithin{equation}{section}

\newcommand{\delete}[1]{} 
\newcommand{\nt}{\noindent}

\def\al{\alpha}

\def\a{\alpha}

\newcommand{\g}{\gamma}

\newcommand{\sk}{\vskip 0.2cm}

\newcommand{\nl}{\newline}
\newcommand{\ben}{\begin{enumerate}}

\newcommand{\een}{\end{enumerate}}
\newcommand{\bit}{\begin{itemize}}

\newcommand{\eit}{\end{itemize}}

\def\R {{\mathbb R}}
\def\N {{\mathbb N}}
\def\Z {{\mathbb Z}}
\def\Q {{\mathbb Q}}

\def\T {{\mathbb T}}

\def\Iso{{\mathrm{Iso}}\,}
\def\Aut{{\mathrm Aut}\,}

\def\nbd {neighborhood }

\def\QED{\nobreak\quad\ifmmode\roman{Q.E.D.}\else{\rm Q.E.D.}\fi}

\begin{document}

\title[]{Orderable groups and semigroup compactifications} 

\dedicatory{Dedicated to my friend Eli Glasner on the occasion of
	his 75th birthday}


\author[]{Michael Megrelishvili}
\address{Department of Mathematics,
	Bar-Ilan University, 52900 Ramat-Gan, Israel}
\email{megereli@math.biu.ac.il}
\urladdr{http://www.math.biu.ac.il/$^\sim$megereli}


\date{2022, August 17}  


\thanks{{\it 2020 AMS classification:}
37B99, 06F15, 06F30, 22A25, 54H15, 46B99}
\thanks{This research was supported by a grant of the Israel Science Foundation (ISF 1194/19) 
 and also by the Gelbart Research Institute at the Department of Mathematics, Bar-Ilan  University}

\begin{abstract} 
	Our aim is to find some new links between linear (circular) orderability of groups and topological dynamics.
	We suggest natural analogs of the concept of algebraic orderability for \textit{topological} groups involving order-preserving actions on compact spaces and the corresponding enveloping semigroups in the sense of R. Ellis.  
	This approach leads to several natural questions. Some of them might be useful also for discrete (countable) orderable groups.  
\end{abstract}

\maketitle  
\setcounter{tocdepth}{1}
 \tableofcontents

\section{Introduction}

Orderability properties of groups is an active research direction since a long time. 
 This theory is closely related to topology. See, for example, \cite{CR-b,DNR,Calegari04,BS}. 	

 Ordered dynamical systems were studied in several recent publications  concerning tame systems, Sturmian like circularly ordered symbolic systems and some dynamical generalizations of the amenability concept. 
 See joint works with Eli Glasner \cite{GM-c,GM-tLN,GM-int,GM-TC} and also \cite{Me-b}. 
Investigation of order-preserving compact dynamical systems provides
 natural framework to study orderable groups (and the orderability itself). 
 We hope that bringing some more dynamical tools, like the Ellis semigroup and Banach representations of actions, to the theory of orderable groups can lead to new interesting lines of research. 
  In Section \ref{s:c}, we give some background about circular orders and necessary results about \textit{circular topology}. 
 Some of these results seem to be new and hopefully have independent interest. 

The following topological version of (left) linear and circular orderability of abstract groups first was defined in \cite{GM-c}. 


\begin{defin}  \label{d:orderly} 
	Let us say that a topological group $G$ is \textit{c-orderly} (\textit{orderly}) if $G$ topologically can be embedded into the topological group $H_+(K)$  of all circular (linear) order-preserving homeomorphisms of $K$, endowed with standard compact open topology, for some compact circularly (resp., linearly) ordered space $K$.  
\end{defin}

Theorems A and C below in 
this section show that this definition provides a natural topological generalization.  

%

\begin{question} \label{q:1} 
	Which topological groups are orderly ? c-orderly ? 
\end{question}

Immediate tautological examples of orderly (c-orderly) Polish groups are 
$H_+(K)$ for (c)-ordered $K$. In particular, $H_+[0,1]$ (and $H_+(S^1)$, where $S^1$ is the circle). The class of orderly groups is closed under the completion (Proposition \ref{p:complLin}) and finite products (Proposition \ref{p:prod}). In particular, the group $\R^n$ is orderly for every $n \in \N$ (Example \ref{e:R^n}). For a stronger result about $\R^n$ see Remark \ref{r:RnAspOrd}. 

According to Remark \ref{r:chech}.1, every linearly ordered set $(X,\leq)$ induces a natural circular order $\circ_{\leq}$ on $X$. Moreover, any linear order preserving transformation $g \in H_+(X,\leq)$ preserves the circular order $\circ_{\leq}$ (i.e, $g \in H_+(X,\circ_{\leq})$). 
So, every orderly group is c-orderly.

An important particular case of a c-orderly (but not orderly) group is 
the Polish group $H_+(S^1)$.  
It algebraically 
contains any countable left c-ordered group. 

By Corollary \ref{c:withPointw} 
for every linearly (circularly) ordered set $(X,\leq)$ ($(X,\circ)$) the topological group of all automorphisms $\Aut(X,\leq)$ ($\Aut(X,\circ)$) with the pointwise topology is orderly (c-orderly). 

If a topological group $G$ is orderly then its discrete copy is orderly 
as it immediately follows from Theorem \ref{t:LinCase}. 
The converse is not true even for the precompact cyclic group $(\Z,d_p)$, where $d_p$ is the $p$-adic metric. The reason is that if a topological group $G$ is orderly 
then it cannot contain a \textit{topological torsion element} (Proposition \ref{p:TopTorsion}). Like, the following standard fact: an abstract left orderable group is torsion-free (see. e.g., \cite{CR-b}).

%

By Proposition \ref{p:noCompactSubgr}, 
another restriction for orderly groups is the absence of nontrivial compact subgroups. Now, Montgomery-Zippin classical theorem implies that 
every locally compact orderly group is a Lie group (Corollary \ref{c:Lie}). 
Note that 
every locally compact subgroup of $H_+(S^1)$ is a Lie group (Ghys \cite[Theorem 4.7]{Ghys}). 

Our results show that a discrete group $G$ is orderly (c-orderly) if and only if $G$ is left linearly (resp., left circularly) orderable. 
More precisely, in Theorem \ref{t:c-ordCase} below, we prove the following left circular orderability criterion (a version of Theorem A for linear orders is proved in Theorem \ref{t:LinCase}).  
\sk

\nt \textbf{Theorem A.}  
\textit{Let $G$ be an abstract group. 
	The following are equivalent:
	\ben 
	\item $G$ is left circularly orderable;   
	\item $(G,\tau_{discr})$ is c-orderly. 
	\een}


In our proof, we use Theorem \ref{t:lim}, about zero-dimensional circular order preserving $G$-compactifications of circularly ordered $G$-sets, using canonically defined inverse limits of finite cycles. In the reverse direction 
we use also a technical sufficient condition of the left orderability, Lemma \ref{l:mainTechn}. Its idea 
 was inspired by the construction of c-ordered lexicographic products. 
 This argument shows also, as a byproduct, that $G$ is left circularly orderable if and only if $G$ admits an 
 effective circular order-preserving action on a circularly ordered set $(X,\circ)$. 
 This result is known from Zheleva \cite{Zheleva97}.

%


\sk 
Recall that the enveloping (Ellis) semigroup $E(K)$ of a continuous action $G \times K \to K$ on a compact space $K$ is the pointwise closure of all $g$-translations $\{\tilde{g} \colon K \to K, x \mapsto gx: \ g \in G\}$ in the compact space $K^K$. This defines the so-called Ellis compactification $j\colon G \to E(K)$. It is an important tool in topological dynamics. See, for example, the survey works \cite{Gl-env} and \cite{GM-survey}. 

In Sections \ref{s:env} and \ref{s:c-env},  
we deal with \textit{ordered enveloping semigroups}. There are several good reasons which inspired us to consider such objects. One of them is a work of Hindmann and Kopperman \cite{HK} where the authors embed any orderable discrete group into a linearly ordered compact right topological semigroup. One of the key ideas was to use the Nachbin's ordered compactifications. 

The second reason comes from our papers \cite{GM-c,GM-tLN,GM-TC} with Eli Glasner, where we study Sturmian like and other circularly ordered dynamical systems. 
In many such cases, the enveloping semigroup is a circularly  ordered semigroup. See Examples \ref{e:exS} below.   

 In the present work, we 
show that enveloping  semigroups of ordered dynamical systems, under natural additional assumptions,  
generate ordered right topological semigroup compactifications of groups. 
See \textit{dynamical orderability} of topological groups (Definition \ref{d:StrongOrderly}).

\begin{question} \label{q:generalRTSC} 
	Which topological groups $G$ admit proper linearly (circularly) order compact right topological semigroup compactification?  We call them: \emph{dynamically (c-)orderable} topological groups. 	
\end{question}

Results of Section \ref{s:env} show that, for discrete groups, these are exactly linearly (circularly) orderable groups in the usual sense. 

\sk 

\nt \textbf{Theorem B.} (see Theorem \ref{t:env}) 
	\textit{Let $K$ be a linearly ordered compact $G$-system 		
		and $\leq$ be a linear order on $G$ such that every orbit map $\tilde{x}\colon G \to K, g \mapsto gx$ is linear order preserving for every $x \in K$.  
	Then the Ellis semigroup $E(K)$ is a linearly ordered semigroup and the Ellis compactification $j\colon G \to E(K)$ is a linearly ordered semigroup  compactification.}   

\sk 
I do not know if the circular analog of Theorem B remains true. 

\sk 

In Theorem \ref{t:bi-Case} we show that  an abstract group $G$ is linearly orderable if and only if   
$G$ is dynamically orderable (admits a linear order semigroup compactification $\g\colon G \hookrightarrow S$ which is a topological embedding of the discrete copy of $G$). This result can be derived also from results of Hindman and Kopperman \cite{HK}. 
For  countable discrete $G$ we may suppose that $S$ is hereditarily separable and first countable. For this topological conclusion we use results of 
Ostaszewski \cite{Ost} (see Remark \ref{r:Ost}).


\sk 
In the following theorem, we use a classical result of Scwierczkowski \cite{Scw} which uses lexicographic products. 

 \sk 
\nt \textbf{Theorem C.} (see Theorem \ref{t:c-Case}) 
Let $G$ be an abstract group. 
The following are equivalent:
\ben 
\item $G$ is circularly orderable;   
\item $G$ admits a c-order semigroup compactification $\g\colon G \to S$ which is a topological embedding  of the discrete group $G$.  
\een


\sk 
Every (c-)ordered compact $G$-space $K$ is \textit{tame} in the sense of A. K\"{o}hler \cite{Ko} \textit{(regular}, in the original terminology).  
 If $K$ is metrizable then it is equivalent to say that the enveloping semigroup $E(K)$ is ``small"; namely, a separable Rosenthal compact space (see \cite{GM-survey,GM-c}). 
In view of a hierarchy of tame metric dynamical systems (see \cite{GM-TC}) induced by the  Todor\u{c}evi\'{c}' Trichotomy 
for Rosenthal compact spaces, we ask the following   
%

\begin{question} \label{q:SmallEnv} 
	Which (c-)orderly topological groups $G$ admit an effective (c-)ordered action 
	on a compact metrizable space $K$ such that the Ellis compactification $G \hookrightarrow E(K)$ is a topological embedding and the enveloping semigroup $E(K)$ 
 is:
 
  a) metrizable? b) hereditarily separable? c) first countable? 
\end{question}

According to results of \cite{GM-c}, 
every (c-)ordered compact $G$-space $K$ is representable on a \textit{Rosenthal Banach space} (not containing an isomorphic copy of $l^1$). In contrast, the actions of the topological groups $H_+[0,1]$ and $H_+(S^1)$ on $[0,1]$ and $S^1$, respectively, 
are not Asplund representable (Theorem \ref{t:notAspOrd}).   
This raises natural questions (\ref{q:AspOrd} and \ref{q:CountAspOrd}) which orderly groups (in particular, countable discrete)  $G$ admit an order preserving action on a compact metrizable ordered space $K$ such that $j\colon G \hookrightarrow E(K)$ is a topological embedding and 
$K$ is an Asplund representable $G$-space in the sense of \cite{GM-survey}. By results of \cite{GMU08} this is equivalent to the metrizability of $E(K)$ (cf. case (a) in Question \ref{q:SmallEnv}). 
For instance, this holds for the group $\R^n$ but it is not true for the topological groups $H_+[0,1]$ and $H_+(S^1)$. 
 It seems to be unclear if every countable discrete ordered group $G$ admits such semigroup compactifications. 

\sk 
\subsection*{List of all questions}   

\ref{q:1}, 
\ref{q:generalRTSC}, 
\ref{q:SmallEnv}, 
\ref{q:LC}, 
\ref{q:c-version}, 
\ref{q:MetrizSemComp}, 
\ref{q:SmallEnv2}, 
\ref{q:AspOrd}, 
\ref{q:CountAspOrd}. 

\sk

\noindent \textbf{Acknowledgment.}  
We thank D. Dikranjan and V. Pestov for helpful discussions.  
We thank the anonymous referee for his many helpful suggestions and corrections.

\sk \sk 
\section{Circular order and its topology} 
\label{s:c} 

\subsection{Linear orders} 
By a linear order $\leq$ on $X$ we mean a reflexive, anti-symmetric, transitive  relation which is totally ordered, meaning that for distinct $a,b \in X$ 
we have exactly one of the alternatives: $a<b$ or $b<a$. As usual, $a < b$ means  that $a \leq b \wedge  a \neq b$.     

For every linearly ordered set $(X,\leq)$ and $a,b \in X$, define the rays 
$$(a,\to):=\{x \in X: a < x\}, \ \ \ \ (\leftarrow,b):=\{x \in X: x <b\}.$$ 
The set of all such rays form a subbase for the standard \emph{interval topology} $\tau_{\leq}$ on $X$.

A map $f\colon X_1 \to X_2$ between two linearly ordered sets $(X_1,\leq_1), (X_2,\leq_2)$ is said to be linearly ordered preserving (in short: LOP) if $a \leq_1 b$ implies $f(a) \leq_2 f(b)$.

Let $\a \colon G \times X \to X$ be a left action of an abstract 
group $G$ on $X$. 
Consider the associated group homomorphism $h_{\a} \colon G \to S(X)$, where $S(X)$ is the symmetric group of all bijections $X \to X$. 
If $h_{\a}$ is injective then this action is said to be {\it effective}. 
Let $(X,\leq)$ be a linearly ordered set. Denote by $\Aut(X,\leq)$ its automorphism group. So, 
$$\Aut(X):=\{g \in S(X): x\leq y \Leftrightarrow gx \leq gy \ \ \forall x,y \in X\},$$
Clearly, $\Aut(X)$ is a subgroup of $S(X)$. 
The action $\a$ preserves the order $\leq$ if and only if $h(G) \subset \Aut(X)$. 

If $K$ is a linearly ordered topological space, then by $H_+(K)$ we denote the group of all order-preserving homeomorphisms. If $K$ is compact, then $H_+(K)$ with the compact open topology is a topological group. A \textit{linearly ordered dynamical $G$-system} $K$ means that $K$ is a compact $G$-space with respect to a given continuous action of $G$ on a compact space $(K,\tau)$, where the topology $\tau$ is the interval topology of some $G$-invariant linear order on $K$.  Equivalently, this means that we have a continuous homomorphism 
$h\colon G \to H_+(K)$. 

\begin{lem} \label{l:min} 
	Let $(K,\leq)$ be a compact linearly ordered $G$-space. 
	Assume that $K$ is minimal (that is, every $G$-orbit is dense in $K$). Then $K$ is trivial. 	
\end{lem}
\begin{proof}
	Since $(K,\leq)$ is a compact linearly ordered space, it has the greatest element $m \in K$. Its $G$-orbit is dense in $K$. On the other hand, $m$ is $G$-fixed because $G$ is a subgroup of $H_+(K,\leq)$. Therefore, $K=\{m\}$ is the singleton. 
\end{proof}

In the following lemma the first assertion is well known 
(see, for example, Nachbin \cite{Nachbin}) and easy to prove directly. 

\begin{lem}  \label{l:LinOrdClosed}  \ 
	\begin{enumerate}
		\item  Every linear order $\leq$ on a set $X$ is closed in the topological square $X \times X$, where $X$ carries the interval topology.  
		\item $H_+(K)$ is a closed subgroup of $H(K)$ for every compact linearly ordered space $K$.
	\end{enumerate}
\end{lem}\begin{proof}
	(1) Let the pair $(a,b)$ does not belong to the relation $\leq$ (which is a subset of $X \times X$). Let $b<a$ (the case of $a<b$ is similar). Then there are two cases. If there exists $c \in X$ such that $b <c<a$ then the open set $(\leftarrow,a) \times (b,\rightarrow)$ contains the point $(a,b)$ and is disjoint from the relation $\leq$. If there is no such $c$ then take $(b,\rightarrow) \times (\leftarrow,a)$. 
	
	(2) By (1) the subgroup $H_+(K)$ is pointwise closed in $H(K)$.  
\end{proof}

\begin{lem} \label{l:DenseLin} 
	Let $X$ be a topological space and $R \subset X \times X$ be a closed partial order on $X$. Suppose that $Y$ is a dense subset of $X$ such that the restricted partial order $R_Y$ on $Y$ is a linear order. Then $R$ is a linear order on $X$.  
\end{lem} 
\begin{proof}
	The union $R \cup R^{-1}$ is closed in $X \times X$. The subset $Y \times Y$ is dense in $X \times X$ and is contained in $R \cup R^{-1}$. Hence, 
	$R \cup R^{-1} = X \times X$.  	
\end{proof}	


\sk  \sk 
\subsection{Circular orders}  
Now we recall some definitions about circular orders. 
In contrast to the interval topology of linear orders, the topology of circular orders, as far we know, hardly is found in the literature. An exception is a book of Kok \cite{Kok}, where such a topology is mentioned episodically. 
For more information and properties we refer to \cite{GM-c,GM-int}. 

\begin{defin} \label{newC} \cite{Kok,Cech,Tararin}  
	Let $X$ be a set. A ternary relation $R \subset X^3$ (sometimes denoted also by $\circ$) on $X$ is said to be a {\it circular} (or, \emph{cyclic}) order  
	if the following four conditions are satisfied. It is convenient sometimes to write shortly $[a,b,c]$ instead of $(a,b,c) \in R$. 
	\ben
	\item Cyclicity: 
	$[a,b,c] \Rightarrow [b,c,a]$;  
	
	\item Asymmetry: 
	$[a,b,c] \Rightarrow (b, a, c) \notin R$; 
	
	\item Transitivity:    
	$
	\begin{cases}
		[a,b,c] \\
		[a,c,d]
	\end{cases}
	$ 
	$\Rightarrow [a,b,d]$;
	
	\item Totality: 
	if $a, b, c \in X$ are distinct, then \ $[a, b, c]$ 
	or $[a, c, b]$. 
	\een
\end{defin}

Observe that under this definition $[a,b,c]$ implies that $a,b,c$ are distinct. 

For 
$a,b \in X$, define the (oriented) \emph{intervals}: 
$$
(a,b)_{\circ}:=\{x \in X: [a,x,b]\}, \  [a,b]_{\circ}:=(a,b) \cup \{a,b\}, \  [a,b)_{\circ}:=(a,b) \cup \{a\}, \ (a,b]_{\circ}:=(a,b) \cup \{b\}. 
$$
Sometimes we drop the subscript, or write $(a,b)_o$ when the context is clear. 
Similarly can be defined $[a,b)$ and $(a,b]$. 
Clearly, $[a,a]=\{a\}$ for every $a \in X$.

\begin{prop} \label{Hausdorff}  \ 
	\ben 
	\item 
	For every c-order $\circ$ on $X$, the family of subsets
	$${\mathcal B}_1:=\{X \setminus [a,b]_\circ : \ a,b \in X\} \cup \{X\}$$  
forms a base for a topology $\tau_{\circ}$ on $X$ which we call the 
\emph{circular topology} (or, the \emph{interval topology}) of $\circ$. 
\item If $X$ contains at least three elements, then the (smaller) family of intervals  
$${\mathcal B}_2:=\{(a,b)_{\circ} : \  a,b \in X, a \neq b\}$$
forms a base for the same topology $\tau_{\circ}$ on $X$. 
\item The circular topology $\tau_\circ$ of every circular order $\circ$ is Hausdorff. 
\een 
\end{prop}

The family ${\mathcal B}_2$ 
and the corresponding topology was mentioned in \cite[p. 6]{Kok} without any assumptions on $X$. However, if $X$ contains only one or two elements then every $(a,b)_\circ$ is empty. In this case ${\mathcal B}_2$ is not a topological base at all. 
As we have seen in Proposition \ref{Hausdorff}, ${\mathcal B}_2$ generates a topology on $X$ under a (minor) assumption that $X$ contains at least 3 elements. 
So, the first assertion of Proposition \ref{Hausdorff} is an accurate form of the definition from \cite{Kok}.

Note that in every circularly ordered set $(X,\circ)$ we have $X \setminus [a,b]_{\circ}=(b,a)_{\circ}$ and $X \setminus (a,b)_{\circ}=[b,a]_{\circ}$ for every distinct $a \neq b$. So the ``closed interval" $[b,a]_{\circ}$ is always closed in the interval topology for all, not necessarily distinct, $a ,b$. 

\sk 


\begin{remark} \label{r:chech} \ E. \v{C}ech \cite[p. 35]{Cech}
\ben 
\item 
Every linear order $\leq$ on $X$ defines a \emph{standard circular order} $\circ_{\leq}$ on $X$ as follows: 
$[x,y,z]$ iff one of the following conditions is satisfied:
$$x < y < z, \ y < z < x, \  z < x < y.$$	
\item (standard cuts) Let $(X,R)$ be a c-ordered set and $z \in X$. 
The relation 
$$z \leq_z x, \ \ \ \  a <_z b \Leftrightarrow [z,a,b] \ \ \ \forall a \neq b \neq z \neq a$$
is a linear order on $X$ and $z$ is the least element. This linear order 
restores the original circular order, meaning that $R_{\leq_z}=R$. 
\een
\end{remark}

\begin{lem} \label{l:closed} 
\cite{GM-c} 	
If $(X,\leq)$ is a linearly ordered set such that its 
circular topology is compact, then the corresponding circular order generates the same topology (the compactness is essential). 
\end{lem}

If $X_1 \to X_2$ is linear order preserving, then it is also c-order preserving in the sense of Definition \ref{d:c-ordMaps}. 
On the set $\{0, 1, \cdots, n-1\}$ consider the standard c-order modulo $n$. 
Denote this c-ordered set, as well as its order, simply by $C_n$. 
Every finite c-ordered set with $n$ elements is isomorphic (Definition \ref{d:c-ordMaps}) to $C_n$.

\begin{defin} \label{d:cycl}  Let $(X,R)$ be a c-ordered set. 
We say that a vector
$(x_1,x_2, \cdots, x_n) \in X^n$ is a \textit{cycle} in $X$ if it satisfies the following two conditions: 

\ben 
\item for every $[i,j,k]$ in $C_n$ and \textit{distinct} 
$x_i, x_j, x_k$ we have $[x_i,x_j,x_k]$; 
\item $x_i=x_k \ \Rightarrow$ \ 
$(x_i=x_{i+1}= \cdots =x_{k-1}=x_k) \ \vee \ (x_k=x_{k+1}=\cdots =x_{i-1}=x_i).$  
\een  
\textit{Injective cycle} means that all $x_i$ are distinct. 
\end{defin}

Using cycles (chains) one may define \textit{bounded variation} functions on a c-ordered set $(X,\circ)$ (resp. on a linearly ordered set $(X,\leq)$). 
This concept helps to study the tameness,  Banach representations and other properties of (c-)ordered dynamical systems. As well as getting a generalized Helly theorem (see \cite{GM-c,GM-TC,Me-Helly}).

\begin{defin} \label{d:c-ordMaps} Let $(X_1,R_1)$ and  $(X_2,R_2)$ be c-ordered sets. 
A function $f \colon X_1 \to X_2$ is said to be {\it c-order preserving}, 
or {\it COP}, if $f$ moves every cycle to a cycle. 
Equivalently, if it satisfies the following two conditions:

\ben 
\item for every $[a,b,c]$ in $X$ and \textit{distinct} 
$f(a), f(b), f(c)$ we have $[f(a), f(b), f(c)]$; 
\item if $f(a)=f(c)$ then $f$ is constant on one of the closed intervals $[a,c], [c,a]$. 
\een  
\end{defin}

In general, condition (1) does not imply condition (2). 
Indeed, consider a 4-element cycle $X=Y=\{1,2,3,4\}$ and a selfmap $p: X \to X, p(1)=p(3)=1, p(2)=p(4)=2$. Then (1) is trivially satisfied because $f(X)<3$ but not (2).

\begin{remark} \label{r:COP} One may show that if $|p(X)| \geq 3$ then 
condition (1) in Definition \ref{d:c-ordMaps} implies (2). We omit the details. 
\end{remark}

A composition of c-order-preserving maps is c-order preserving. 
We let $M_+(X_1,X_2)$ be the collection of c-order-preserving
maps from $X_1$ into $X_2$.	 
$f$ is an \textit{isomorphism} if, in addition, $f$ is a bijection (in this case, of course, only (1) is enough).  
Denote by $\Aut(X)$ the group of all COP automorphisms $X \to X$ which is a subgroup of the symmetric group $S(X)$ of all bijections $X \to X$. 

\sk 
For every circularly ordered set $X$ and every subgroup $G \subset \Aut(X)$, the corresponding action $G \times X \to X$ defines a circularly ordered $G$-set $X$. The definitions of the group 
$H_+(K,\circ)$ and \textit{circularly ordered compact dynamical $G$-system} are also understood for every compact $K$.  


\begin{lem} \label{l:c-is-Open} \cite{GM-c}
Let $\circ$ be a circular order on $X$ and $\tau_{\circ}$ the induced circular (Hausdorff) topology. 
Then for every  $[a,b,c]$ there exist neighborhoods $U_1,U_2,U_3$ of $a,b,c$ respectively such that $[a',b',c']$ for every $(a',b',c') \in U_1 \times U_2 \times U_3.$  
\end{lem} 

\begin{cor} \label{c:subgr+completion} \
\begin{enumerate}
\item Let $K$ be a circularly ordered compact space. Then $H_+(K)$ is a closed subgroup of the homeomorphism group $H(K)$. 
\item The completion of a c-orderly topological group $(G,\tau)$ is c-orderly. 
\end{enumerate} 	
\end{cor}
\begin{proof}
Straightforward using Lemma \ref{l:c-is-Open}. 
\end{proof}

\sk 
\begin{defin} \label{d:c-convex} 
Let $(X,R)$ be a circularly ordered set. Let us say that a subset $Y$ in $X$ is \textit{convex} in $X$ if for every $a,b \in Y$ at least one of the intervals $[a,b]$, $[b,a]$ is a subset of $Y$. 
\end{defin}

According to Definition \ref{d:c-convex} exactly the following subsets of $(X,\circ)$ are convex:
$$\{\emptyset, X, (u,v)_{\circ}, [u,v]_{\circ}, (u,v]_{\circ}, [u,v)_{\circ}: \ u,v \in X\}.$$

Now the second condition in Definition \ref{d:c-ordMaps} can be reformulated as follows: the preimage $f^{-1}(c)$ of a singleton $\{c\}$ is  convex. 

Similar to the well known case of \textit{generalized ordered} (GO) spaces (see, for example, \cite[p. 457]{Nagata}), one may define \textit{generalized circularly ordered} (GCO) spaces as topological subspaces of circularly ordered spaces. Then for (GCO), convex topologically open subsets and lexicographic products play the similar roles as for (GO). 

%
%
%
%

\sk 
\subsection{Lexicographic order}

\label{r:prod} 

For every c-ordered set $C$ and a linearly ordered set $L$, one may define the so-called 
\emph{c-ordered lexicographic product} $C \otimes L$. See, for example, \cite{CJ} and also Figure 1.   
\begin{figure}[h]
\begin{center} \label{F1}
\scalebox{0.3}{\includegraphics{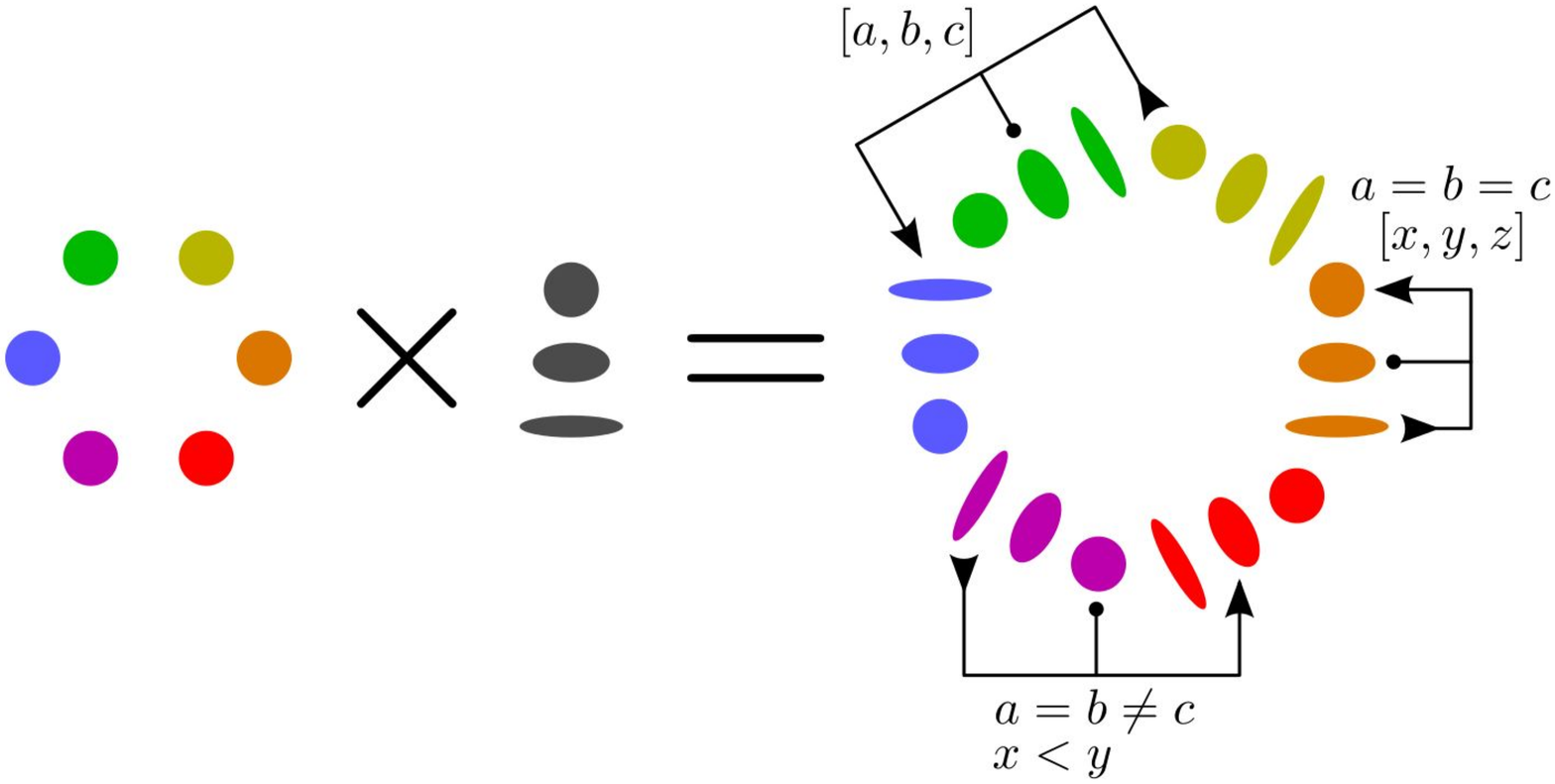}}
\caption{c-ordered lexicographic product (from Wikipedia - \textit{Cyclic order})}
\end{center}
\end{figure} 

\begin{defin} \label{d:lexic}  
More formally, let $(a,x), (b,y), (c,z)$ be distinct points of $C \times L$. Then 
$[(a,x), (b,y), (c,z)]$ in $C \otimes L$ will mean that one of the following conditions is satisfied:
\begin{enumerate}
\item $[a, b, c]$.
\item $a = b \neq c$ and $x < y$.
\item $b = c \neq a$ and $y < z$.
\item $c = a \neq b$ and $z < x$.
\item $a = b = c$ and $[x, y, z]$ (in the cyclic order on $L$ induced by the linear order). 
\end{enumerate} 
\end{defin}

\sk
\subsection{Compactness of the circular topology}

Our aim is to obtain  a natural circular version of the following classical result.

\begin{f} \label{f:LinComp} 
The interval topology on $X$ of a linear order $\leq$ is compact if and only if every subset of X has a supremum (including that $\sup(\emptyset)=\min(X)$ must exist).
\end{f} 

The following result is a particular case of \cite[Lemma 2.1]{Kem} proved by N. Kemoto. It can be obtained also directly using Fact \ref{f:LinComp}.

\begin{f} \label{f:Kem} \cite{Kem} 
A lexicographic linearly ordered product $Y \otimes_l L$ of a compact linearly ordered space $Y$ and a compact linearly ordered space $L$ is compact in the interval topology.  
\end{f}

Following Novak \cite{Novak-cuts}, we define cuts and gaps in c-ordered sets. 
Let $(X,\circ)$ be a c-ordered set. A linear order $\leq$ on $X$ is said to be a {\it cut} if 
$$
a < b < c \ \text{in} \ (X, \leq) \ \text{implies that} \  [a,b,c] \ \text{in} \ (X,\circ). 
$$
A {\it gap} on $(X,\circ)$ is a cut $(X, \leq)$ such that it has neither the least nor the greatest element. This continues the idea of classical Dedekind cuts for linear orders. 

\begin{lem}
\label{l:cut1} 
Let $\leq$ be a cut on $(X,\circ)$.
\ben 
\item \cite[Lemma 2.2]{Novak-cuts} If 
$[a,b,c]$ then either 
$a<b<c$ or $b<c<a$ or $c < a<b$. 
\item \cite[Theorem 2.5]{Novak-cuts} Every linear order $\leq_z$ (Remark \ref{r:chech}) is a (standard) cut on $(X,\circ)$ for every $z \in X$.
\een
\end{lem}

\begin{lem} \label{l:cut} 
Let $\leq$ be a cut on $(X,\circ)$.
\ben 

\item 
$(a,b)_{\leq}=(a,b)_{\circ}$, $[a,b]_{\leq}=[a,b]_{\circ}$,  $[a,b)_{\leq}=[a,b)_{\circ}$, $(a,b]_{\leq}=(a,b]_{\circ}$,   for every $a<b$. 
\item For every interval $[a,b]_{\circ}$ the subspace topology inherited from $\tau_{\circ}$ is the same as the interval topology $\tau_{\leq_a}$ of $\leq_a$. The same is true for other convex intervals $[a,b)_\circ$, $(a,b)_\circ$, $(a,b]_{\circ}$. 
\een
\end{lem}
\begin{proof}
(1) It is enough to show that	$(a,b)_{\leq}=(a,b)_{\circ}$. 

Let $x \in (a,b)_{\leq}$. Then by definition of cut we have $[a,x,b]$. Hence, $x \in (a,b)_{\circ}$. 

Conversely, let $x \in (a,b)_{\circ}$. Then by Lemma \ref{l:cut1}.1, either 
$a<x<b$ or $x<b<a$ or $b < a<x$. By our assumption $a<b$. So, we necessarily have $a<x<b$. Hence, $x \in (a,b)_{\leq}$.

(2) 
First note that 
$$\forall x \in [a,b]_{\circ} \ \ (-\infty,x)_{\leq_a}=[a,x)_{\leq_a}=(b,x)_{\circ} \cap [a,b]_{\circ}.$$
This shows that the circular topology contains the interval topology. 
In order to see the converse inclusion note that 
$$
(c,d)_\circ \cap [a,b]_{\circ}=(c,d)_{\leq_a} \cap [a,b]_{\circ} \ \text{if} \ [a,c,d]
$$
and 
$$
(c,d)_\circ \cap [a,b]_{\circ} =(c,b]_{\leq_a} \cup [a,d)_{\leq_a} \ \text{if} \ [a,d,c]. 
$$

For other convex subsets the proof is similar. 
\end{proof}

\begin{thm} \label{t:c-comp} 
Let $\tau_{\circ}$ be the circular topology of a circular order on a set $X$. The following conditions are equivalent:
\ben 
\item $\tau_{\circ}$ is a compact topology. 
\item $(X,\circ)$ is complete in the sense of Novak \cite{Novak-cuts} (no gaps in $(X,\circ)$).  
\item $[a,b]_{\circ}$ is compact in the subspace topology of $(X,\tau_{\circ})$ for every $a,b \in X$.
\een 
\end{thm}
\begin{proof}
(1) $\Rightarrow$ (2) 
Let $(X,\leq)$ be a gap. Then 
$\cup \{(a,b)_{\leq}: a < b\}=X$ has no finite subcovering. This cover is open in the circular topology  because  $(a,b)_{\leq}=(a,b)_{\circ}$ for every $a<b$ by Lemma \ref{l:cut}.2. 

\sk 
(2) $\Rightarrow$ (3)   
Assume that 
$[a,b]_{\circ}$ is not compact in its subspace topology. Then $([a,b]_{\leq_a},\leq_a)$ 
has a subset $A$ such that $\sup(A)$ does not exist (use Fact \ref{f:LinComp}). Then $([a,b]_{\leq_a}, \leq_a)$ has a ``linear gap", a subset $A \subset [a,b]_{\leq_a}$ without supremum. Then $A$ does not have supremum also as a subset of $(X_{\leq_a},\leq_a)$. 
Let  $X_{1}:=\{x \in X: \forall a \in A \ \ a<x\}$, $X_2:=X \setminus X_1$. 
Define now a linear order $\leq_A$ by the following rule 
$x_1 <_A x_2$ for every $x_1 \in X_1, x_2 \in X_2$ keeping the old relation for other possible pairs $(x,y) \in X=X_1 \cup X_2$. Then $\leq_A$ is a cut on $X$ which is a gap for $(X,\circ)$.

\sk 
(3) $\Rightarrow$ (1) 
For every $a \neq b$ we have 
$X=[a,b]_{\circ} \cup [b,a]_{\circ}$. Hence, $X$ is compact being a union of two compact subsets. 
%
%
%
%
\end{proof}

Note that every circularly ordered set admits a completion \cite{Novak-cuts}.

\begin{prop} \label{p:LexComp}  
A lexicographic c-ordered product $K \otimes_c L$ of compact c-ordered space $K$ and a compact linearly ordered space $L$ is a compact c-ordered space.  
\end{prop}
\begin{proof} Let $u \leq v$ in $L$ and $a,b \in K$.  By Theorem \ref{t:c-comp}, we have to show that the interval  $[(a,u),(b,v)]_{\circ}$ of $K \otimes_c L$ is compact. 

Consider the linearly ordered set $[a,b]_{\leq_a}$. It is compact by Theorem  \ref{t:c-comp} (since $K$ is compact) in its circular or linear topologies. Observe that $[(a,u),(b,v)]_{\circ}=[(a,u),(b,v)]_{\leq}$, where $\leq$ is the linear order inherited from the lexicographic linearly ordered product $Y:=([a,b]_{\leq_a} \otimes_l L$. The latter is compact by 
Fact \ref{f:Kem}. Then $[(a,u),(b,v)]_{\circ}$ is also compact (being a closed subset in $Y$). 
\end{proof}

\sk 
\section{Ordered compactifications via inverse limits}

A \textit{compactification} of a topological space $X$ is a 
continuous dense map $j \colon X \to Y$ such that $Y$ is compact and Hausdorff. \textit{Proper compactification} will mean that $j$ is a topological embedding. 
We say an \textit{order compactification} if $X$ and $Y$ are linearly ordered topological spaces and $j$ is order preserving. \textit{Proper order compactification} will mean that the  compactification $j \colon X \to Y$ is both topological and order embedding.


\begin{lem} \label{invLimCirc} \cite{GM-int}
Let $X_{\infty}:=\underleftarrow{\lim} (X_i,I)$ be the inverse limit of the inverse system 
$$\{f_{ij} \colon X_i \to X_j, \ \ i < j, \ i,j \in I\}$$ where $(I,<)$ 
is a directed partially ordered set. Suppose that every $X_i$ is a c-ordered set with the c-order $R_i \subset X_i^3$ and each bonding map $f_{ij}$ is
c-order preserving. On the inverse limit $X_{\infty}$, define a ternary relation $R$ as follows. 
An ordered triple $(a,b,c) \in X_{\infty}^3$ belongs to $R$ iff $[p_i(a),p_i(b),p_i(c)]$ is in $R_i$ for some $i \in I$.  
\begin{enumerate}
\item Then $R$ is a c-order on $X_{\infty}$ and each projection 
map $p_i \colon X_{\infty} \to X_i$ is c-order preserving.
\item \label{invLimCircTop} Assume in addition that every $X_i$ is a compact c-ordered 
space and each bonding map $f_{ij}$ is continuous.  
Then the topological inverse limit $X_{\infty}$ is also a c-ordered 
(nonempty) 
compact space.
\end{enumerate}

\end{lem}

\sk

A particular case of the following construction and its proof (with minor changes) can be found in \cite{GM-int}. 

\begin{thm} \label{t:lim} 
Let $(X,R)$ be a c-ordered set and $G$ is a subgroup of $\Aut(X)$ with the pointwise topology. 
Then there exist: a c-ordered compact zero-dimensional space $X_{\infty}$ such that 
\ben 
\item $X_{\infty}=\underleftarrow{\lim} (X_F, I)$ is the inverse limit of finite c-ordered sets $X_F$, where $F \in I=P_{fin}(X)$. 

\item $X_{\infty}$ is a compact c-ordered $G$-space and $\nu \colon X \to X_{\infty}$ is a dense topological $G$-embedding of the discrete set $X$ such that $\nu$ is a c-order-preserving map.  

\item If $X$ is countable then $X_{\infty}$ is a metrizable compact space. 
\een 
\end{thm}
\begin{proof}  
Let $F:=\{t_1,t_2, \cdots, t_m\}$ be an $m$-cycle on $X$. That is, a c-order-preserving injective map 
$F \colon C_m \to X$, where $t_i=F(i)$ and $C_m:=\{1,2, \cdots, m\}$ with the natural circular order.  We have a natural equivalence ``modulo-$m$" between $m$-cycles (with the same support). 

For every given cycle $F:=\{t_1,t_2, \cdots, t_m\}$,  
define the corresponding finite disjoint covering $cov_F$ of $X$, by adding to the list: all 
points $t_i$ and nonempty intervals $(t_i,t_{i+1})_o$ between the cycle points. More precisely, we consider the following disjoint cover which can be thought of as an equivalence relation on $X$  
$$
cov_F:=\{t_1, (t_1,t_2)_o, t_2, (t_2,t_3)_o, \cdots, t_m, (t_m,t_1)_o \}. 
$$ 
Moreover, $cov_F$ naturally defines also a finite c-ordered set $X_F$ by ``gluing the points" of the nonempty interval $(t_i,t_{i+1})_o$ for each $i$. 
So, the c-ordered set $X_F$ is the factor-set of the equivalence relation $cov_F$ and it contains at most $2m$ elements. 
In the extremal case of $m=1$ (that is, for $F=\{t_1\}$), we define $cov_F:=\{t_1, X \setminus\{t_1\}\}$. 

We have 
the following canonical c-order preserving 
onto
map 

\begin{equation} \label{projection} 
\pi_F \colon X \to X_F, \  \ 
\pi_F(x) =
\begin{cases}
	t_i & {\text{for}} \ x=t_i\\
	(t_i,t_{i+1})_o & {\text{for}} \ x \in (t_i,t_{i+1})_o. 
\end{cases}
\end{equation}

\begin{lem} \label{Claim1}
The family 
$
\{cov_F\}
$ 
where $F$ runs over all finite injective cycles $$F \colon \{1,2,\cdots,m\} \to X$$ on $X$ is a basis of a natural precompact  uniformity $\mu_X$ of $X$.
\end{lem}
\begin{proof} 
Let $Cycl(X)$ be the set of all finite injective cycles.  
Every finite $m$-element subset $A \subset X$ defines a cycle $F_A \colon \{1,\cdots,m\} \to X$  (perhaps after some reordering) which is uniquely defined up to the natural cyclic equivalence and the image of $F_A$ is $A$. 

$Cycl(X)$ 
is a poset if we define $F_1 \leq F_2$ whenever $F_1 \colon C_{m_1} \to X$ is a \emph{sub-cycle} of $F_2 \colon C_{m_2} \to X$. This means that $m_1 \leq m_2$ and $F_1(C_{m_1}) \subseteq F_2(C_{m_2})$.   
This partial order is directed. Indeed, for $F_1,F_2$, we can consider $F_3=F_1 \bigsqcup F_2$ whose support is the union of the supports of $F_1$ and $F_2$. 

For every $F \in Cycl(X)$, we have the disjoint finite 
$\mu_X$-uniform covering $cov_F$ of $X$. 
As before, we can look at $cov_F$ as a c-ordered (finite) set 
$X_F$. Also, as in equation \ref{projection}, we have a c-order-preserving natural map $\pi_F \colon X \to X_F$ which is uniformly continuous into the finite (discrete) uniform space $X_F$. 
Moreover, 
if $F_1 \leq F_2$ then $cov_{F_1} \subseteq cov_{F_2}$. This implies that the equivalence relation $cov_{F_2}$ is sharper than $cov_{F_1}$.  
We have a c-order-preserving (continuous) onto bonding map $f_{F_2,F_1} \colon X_{F_2} \to X_{F_1}$ between finite c-ordered sets. Moreover, $f_{F_2,F_1} \circ \pi_{F_2}=\pi_{F_2}$. 

In this way, we get an inverse system 
$$
f_{F_2,F_1} \colon X_{F_2} \to X_{F_1}, \ \ F_1 \leq F_2 
$$
where $(I,\leq)=Cycl(X)$ is the directed poset defined above. 
It is easy to see that $f_{F,F}=id$ and $f_{F_3,F_1}=f_{F_2,F_1} \circ F_{F_3,F_2}$.

Denote by $$X_{\infty}:=\underleftarrow{\lim} (X_F, \ I)  \subset \prod_{F \in I} X_F$$ the corresponding inverse limit. 
Its typical element is 
$\{(x_F) : F \in Cycl(X)\} \in X_{\infty}$, where $x_F \in X_F$.  The set 
$X_{\infty}$ carries a circular order as in Lemma \ref{invLimCirc}. 

\sk 
\textbf{Definition of the action} $G \times X_{\infty} \to X_{\infty}$.

For every given $g \in G$ (it is c-order preserving on $X$), we have the induced isomorphism $X_F \to X_{gF}$ of c-ordered sets, where $t_i \mapsto gt_i$ and $(t_i,t_{i+1})_o \mapsto (gt_i,gt_{i+1})_o$ for every $t_i \in cov_F$. 
For every 
$F_1 \leq F_2$, we have $f_{F_1,F_2} (x_{F_2})=x_{F_1}$. This implies that $f_{gF_1,gF_2} (x_{gF_2})=x_{gF_1}$. So, $(gx_F)=(x_{gF}) \in X_{\infty}$. 

Therefore, $g \colon X \to X$ can be extended canonically to a 
map 
$$
g_{\infty} \colon X_{\infty} \to X_{\infty}, \ \ g_{\infty} (x_F) := (x_{gF}).
$$ 
\nt This map is a c-order automorphism. Indeed, if $[x,y,z]$ in $X_{\infty}$ then there exists $F \in I$ such that $[x_F,y_F,z_F]$ in $X_F$. Since $g \colon X \to X$ is a c-order automorphism, we obtain that $[gx_{F},gy_{F},gz_{F}]$ in $X_{gF}$. 

For every cycle $F$ in $X$ and the stabilizer subgroup $St(F) \subset G$ we have 
$gx$ and $x$ are in the same element of the basic disjoint covering $cov_F$. It follows easily from this fact that the action $G \times X_{\infty} \to X_{\infty}$ 
is continuous. Here $X_{\infty}$ carries the compact topology of the inverse limit as a closed subset of the topological product $\prod_{F \in I} X_F$ of finite discrete spaces $X_F$.  
\end{proof}

\begin{lem} \label{l:Limit} 
$X_{\infty}$ is a compact c-ordered $G$-space and $i\colon X \to X_{\infty}$ is a dense c-order-preserving embedding with discrete $i(X)$.  
Furthermore, if $X$ is countable then $X_{\infty}$ is a metrizable compact. 
\end{lem}

\begin{proof}  $(X,\mu_X)$ can be treated as the weak uniformity with respect to the family of maps $\{\pi_F\colon X \to X_F: \ F \in Cycl(X)\}$ (into the finite uniform spaces $X_F$). The corresponding topology of $i(X)$ is discrete. 

Observe that $f_{F_2,F_1} \circ \pi_{F_2}=\pi_{F_1}$ for every $F_1 \leq F_2$.  By the universal property of the inverse limit, we have the canonical uniformly continuous map $\pi_{\infty}\colon X \to X_{\infty}$. It is easy to see that it is an embedding of uniform spaces and 
that $\pi_{\infty}(X)$ is dense in $X_{\infty}$. Since $X$ is a precompact uniform space, we obtain that its uniform completion is a compact space and can be identified with $X_{\infty}$. The uniform embedding $X \to X_{\infty}$ is a $G$-map. It follows that the uniform isomorphism $\widehat{X} \to X_{\infty}$ is also a $G$-map. 

On the other hand, this inverse limit $X_{\infty}$ is c-ordered as follows from Lemmas \ref{invLimCirc} and \ref{invLimCircTop}.  Furthermore, as we have
already mentioned the action of $G$ on $X_{\infty}$ is c-order preserving. 
Therefore, $X_\infty$ is a compact c-ordered $G$-system. 
\sk

This proves (1) and (2). In order to prove (3), observe that if $X$ is countable then we have countable many cycles $F$. Therefore, $X_{\infty}$ is metrizable. 
\end{proof}

This completes the proof of Theorem \ref{t:lim}. 
\end{proof}

\sk \sk 
\begin{remark} \label{r:M(G)} 
Theorem \ref{t:lim} was used recently in \cite{GM-int} to prove that 
for cyclically ultrahomogeneous $G$-spaces $X$ 
the c-ordered compact $G$-space $X_{\infty} \setminus X$ is the \textit{universal minimal $G$-system} $M(G)$. In particular, for $X=(\Q,\circ)$, rationals on the circle (in fact, for $\Q/\Z$) and the group $G=\Aut(\Q,\circ)$ with its pointwise topology, the corresponding c-ordered compactum $M(G)=X_{\infty} \setminus X$ looks as the circle after splitting all rational points. 
\end{remark}

\sk 


\begin{thm} \label{t:limLin} 
Let $(X,\leq)$ be a linearly ordered set and $G$ be a subgroup of $\Aut(X)$ with the pointwise topology. 
Then there exists a linearly ordered compact zero-dimensional space $X_{\infty}$ such that 
\ben 
\item $X_{\infty}=\underleftarrow{\lim} (X_F, I)$ is the inverse limit of finite linearly ordered sets $X_F$, where $F \in I=P_{fin}(X)$. 

\item $X_{\infty}$ is a compact linearly ordered $G$-space and 
$\nu \colon X \hookrightarrow X_{\infty}$ is a dense topological $G$-embedding of the discrete set $X$ such that $\nu$ is a LOP map.  
\item If, in addition, $\leq_G$ is a linear order on $G$ such that orbit maps 
$\tilde{x} \colon G \to X$ ($x \in X$) are LOP then all orbit maps $\tilde{a} \colon G \to X_{\infty}$ ($a \in X_{\infty}$) are LOP. 
\item If $X$ is countable then $X_{\infty}$ is a metrizable compact space. 
\een 
\end{thm}
\begin{proof}
The assertions (1) and (2) can be proved similar to Theorem \ref{t:lim}. 
We consider finite chains $F:=\{t_1,t_2, \cdots, t_m\}$ (instead of cycles). The corresponding covering 
$cov_F$ has the form 
$$
cov_F:=\{(-\infty,t_1), t_1, (t_1,t_2), t_2, (t_2,t_3), \cdots, (t_{m-1}, t_m), , t_m, (t_m,\infty)\}, 
$$ 
where we remove all possible empty intervals. 

(3)  Since $\nu(X)$ is dense in $X_{\infty}$, there exists a net $x_{\gamma}$, $\gamma \in \Gamma$ in $X=\nu(X) \subset X_{\infty}$ such that $a= \lim x_{\gamma}$ in $X_{\infty}$. If $g_1 \leq_G g_2$ in $G$ then $g_1x_{\gamma} \leq g_2x_{\gamma}$ in $X$. Hence, also in $\nu(X)$ and $X_{\infty}$. The linear order in every linearly ordered space (in particular, in $X_{\infty}$) is closed (see Lemma \ref{l:LinOrdClosed}.1). Therefore, if $g_1 \leq_G g_2$ in $G$ then we obtain $g_1a= \lim g_1x_{\gamma} \leq \lim g_2x_{\gamma}=g_2a$.  

(4) Similar to Theorem \ref{t:lim}, we have countably many chains $F$ for countable $X$. 
\end{proof}

\sk 
\begin{cor} \label{c:withPointw} 
Let $(X,\leq)$ ($(X,\circ)$) be a linearly (circularly) ordered set and $G$ be a  subgroup of $\Aut(X,\leq)$ ($\Aut(X,\circ)$) with the pointwise topology. Then $G$ is an orderly (c-orderly) topological group.  
\end{cor}

\begin{remark} \label{r:Q} 
Let $X=(\Q,\leq)$ be the rationals with the usual order but equipped with the discrete topology. 	Consider the automorphism group $G:=\Aut(\Q,\leq)$ 
with the pointwise topology.
One may apply Theorem \ref{t:limLin} getting the linearly ordered $G$-compactification 
$\nu \colon X \hookrightarrow X_{\infty}$ (where $X_{\infty}$ is metrizable and zero-dimensional). This compactification, in this case, has a remarkable property (as we show in \cite{Me-Sm100}). Namely, 
$\nu$ is the \textit{maximal $G$-compactification} for the $G$-space $\Q$ (where $\Q$ is discrete). The same is true for every dense subgroup $G$ of $\Aut(\Q,\leq)$ (e.g.,  for Thompson's group $F$). 

Similar result is valid for the circular version. Namely, the rationals on the circle with its circular order
$X=(\Q / \Z,\circ)$, the automorphism group $\Aut(\Q / \Z,\circ)$ and its dense subgroups $G$ (for instance, Thompson's circular group $T$).
\end{remark}

\sk \sk 	
\section{Linearly ordered groups} \label{s:Left}


Linear left orderability is an important property of groups and was  extensively studied in many publications; see for example \cite{CR-b,Ghys,DNR} and references there.

Let $G$ be a group and $\leq$ be a linear order on the set $G$. We say that this order is 
\textit{left invariant} if the left action of $G$ on itself preserves $\leq$. Meaning that 
$x \leq y$ iff $gx \leq gy$ for all $g,x,y \in G$.  
A \textit{right invariant} and \textit{bi-invariant order} can be defined similarly.

\sk  
A group $G$ is \textit{left linearly orderable} iff there exists a linear order $\leq$ on $G$ such that the standard left action of $G$ on itself preserves the order. For simplicity we shortly say $G$ is L-Ord. In the case that left and right action both are order preserving with respect to the same order on $G$, we say that $G$ is \textit{orderable}; abbreviation: Ord.

 L-Ord groups are torsion free. This condition is necessary but not sufficient.  
The classes L-Ord and Ord are surprisingly large, \cite{CR-b,DNR}.
Among others, all free groups and all free abelian groups are Ord. 
 


\sk  
A topological group $G$ is said to be (\textit{Raikov}) \textit{complete} if it is complete with respect to its two-sided uniform structure $\mathcal{U_{LR}}$. 
Recall 
(see, e.g., \cite[Theorem 3.6.10]{AT}) 
 that the completion $r \colon G \hookrightarrow \widehat{G}$ of every topological group $G$ with respect to $\mathcal{U_{LR}}$ naturally admits a structure of a topological group. 
   
  \sk

   Orderly topological groups were defined in Definition \ref{d:orderly} as topological subgroups of $H_+(K)$ for some compact linearly ordered space. 
  		For every linearly ordered set $(X,\leq)$ the topological group $\Aut(X,\leq)$ with its pointwise topology with respect to the discrete set $X$ is orderly (Corollary \ref{c:withPointw}).

\begin{prop} \label{p:complLin} 
	The completion of an orderly topological group $(G,\tau)$ is orderly. 
\end{prop}
\begin{proof}
	Let $(G,\tau)$ be an orderly topological group. Then there exists a compact linearly ordered space $(K,\leq)$ such that $G$ is a topological subgroup of $H_+(K)$. It is well known that the topological group $H(K)$ (of all homeomorphisms) in its standard compact open topology is complete for every compact space $K$. 
	By Lemma \ref{l:LinOrdClosed}.2, $H_+(K)$ is a closed subgroup of $H(K)$. Therefore, $H_+(K)$ is complete, too. Then the closure of $G$ in $H_+(K)$ (which is its completion) is also complete. 
\end{proof}

%
%


The equivalence of (1) and (3) in Theorem \ref{t:LinCase} is well known.  As well as, the fact that a \textbf{countable} group $G$ is L-Ord iff it faithfully acts on $\R$ by order-preserving homeomorphisms (that is, $G$ algebraically is embedded into $H_+(\R)$), \cite{DNR}. 

\begin{thm} \label{t:LinCase} 
	Let $G$ be an abstract group. 
	The following are equivalent:
	\ben 
	\item $G$ is L-Ord; 
	\item $(G,\tau_{discr})$ is orderly \nl (i.e., a discrete copy of $G$ topologically is embedded into the topological group $H_+(K)$ for some compact linearly ordered topological space $(K,\leq)$); 
	\item $G$ algebraically is embedded into the group $\Aut(X,\leq)$ for some linearly ordered set $(X,\leq)$.  
	\een
	
		In (2) we can suppose, in addition, that $\dim K=0$.
\end{thm}  
\begin{proof}
	%
	%
	(1) $\Rightarrow$ (2) 
	One may use the compactification $\nu \colon X \hookrightarrow X_{\infty}$ from Theorem \ref{t:limLin}, where $G=X$, $K=X_{\infty}$ and $\dim X_{\infty}=0$.   
%
	
	(2) $\Rightarrow$ (3) 
	Trivial. 
	
	(3) $\Rightarrow$ (1)  The well-known proof (see, for example, \cite{CR-b,DNR}) is to use a \emph{dynamically lexicographic order} on $G$.
\end{proof}

%

\sk 
Recall that an element $g \in G$ is said to be \textit{torsion} if there exists $n \in \N=\{1,2, \cdots\}$ such that $g ^n=e$. It is well known and easy to prove that in a left orderable group $G$ only the neutral element $e$ is torsion (see, for example, \cite[Prop. 1.3]{CR-b}). One may give a topological analog of this observation for orderly groups.  

\begin{defin} \label{d:tt} 
	Let $G$ be a topological group. We say $g \in G$ is \textit{weakly topologically torsion} (abbr.: wtt) 
	if $e$ belongs to the topological closure $cl(\{g^n: n \in \N\})$  of the set $\{g^n: n \in \N\}$ in $G$.  
\end{defin}

This definition is close to several forms of topological torsion elements known in the literature. See, for example, the survey paper by D. Dikranjan \cite{Dikr-tt} and the book of Dikranjan--Prodanov--Stoyanov \cite[Section 4.4]{DPS89}.  

\begin{prop} \label{p:TopTorsion}
	Let $G$ be an orderly topological group. Then the neutral element is the only weakly topologically torsion element in $G$. 
\end{prop}
\begin{proof}
	Assuming the contrary, let $g \in G, g \neq e$ such that $e \in cl(\{g^n: n \in \N\})$.  
	Since $G$ is orderly, there exists a compact ordered $G$-space $(K,\leq)$ such that $G$ is a topological subgroup of $H_+(K)$. Since $g \neq e$ and the action is effective, there exists $a \in K$ such that $a \neq ga$. Let, $a < ga$ (the second case of $ga < a$ is, of course, similar). Since the action is order preserving, we obtain
	$$
	\min(K) < a < ga < g^2a < \cdots < g^na < \cdots 
	$$
	where $\min(K)$ is the minimal element in $K$ (which always exists in every compact linearly ordered space and is necessarily $H_+(K)$-fixed). 
	By the (pointwise) continuity of the action, there exists a neighborhood $U$ of $e$ in $G$ such that $\min(K)< ua < ga$ for every $u \in U$. 
	There exists $n_0 \in \N, n_0 >1$ such that $g^{n_0} \in U$. Then $\min(K) < g^{n_0}a < ga$. This contradicts the condition $ga < g^{n_0}a$. 
\end{proof}

\begin{cor} \label{c:p-adicsNotOrd} 
	The topological group $G=(\Z,d_p)$ of all integers with the $p$-adic metric is not orderly. 
\end{cor}
\begin{proof}
	In the additive topological group $(\Z,d_p)$ every element $a \in \Z$ is weakly topologically torsion because $\lim p^n a=0$. 
\end{proof}

\sk 
\begin{remark} \label{r:p-adicsNotOrd}  \ 
	\begin{enumerate}
		\item D. Dikranjan informed us that Definition \ref{d:tt} can be reformulated as follows: an element $g \in G$ is wtt if and only if  
		the cyclic subgroup $\langle g\rangle$ of $G$ is either finite or infinite and non-discrete. 
		This immediately implies that in every orderly topological group $G$ all cyclic subgroups are necessarily discrete
		and infinite (essentially strengthens Corollary \ref{c:p-adicsNotOrd}). 
		It seems that 
		the concept of wtt elements in topological groups has its own interest and deserves to be investigated in more details.  
		\item Every orderly topological group $(G,\tau)$ is orderly as an abstract discrete group. Hence, $G$ is L-Ord. The converse, as expected, is not true, as it follows from Corollary \ref{c:p-adicsNotOrd}. 
	\end{enumerate} 
\end{remark}

\sk 
The following result was suggested by the referee with a different proof. 

\begin{prop} \label{p:noCompactSubgr} 
	Let $G$ be an orderly topological group. Then every compact subgroup of $G$ is trivial. 
\end{prop}
\begin{proof} 
	$G$ is a subgroup of $H_+(K,\leq)$ for some compact linearly ordered space $K$. 
Let $G_1$ be a compact subgroup of $G$. Then for every $x \in K$ the orbit $G_1x$ is a minimal $G_1$-space. By Lemma \ref{l:min}, we get $G_1x=\{x\}$. Therefore, $G_1$ is trivial (because the action is effective). 	
	\end{proof} 


\begin{cor} \label{c:Lie} 
	Every locally compact orderly group is a Lie group. 
\end{cor}
\begin{proof}
	By Montgomery-Zippin classical theorem \cite{MZ}, a locally compact group is a Lie group if and only if there is a neighbourhood of 
	the identity which does not contain a non trivial compact subgroup. So, we can apply Proposition \ref{p:noCompactSubgr}.  
	\end{proof}

\sk 
\begin{prop} \label{p:prod} Finite product of orderly groups is orderly. 
\end{prop}
\begin{proof} 
	Let $G_1, G_2$ be orderly topological groups. We have to show that $G_1 \times G_2$ is orderly.
	Let $G_i$ be a topological subgroup of $H_+(K_i,\leq_i)$, where $K_i$ is a linearly ordered compact space. It is enough to show that $G:=G_1 \times G_2$ is topologically embedded into $H_+(K,\leq)$, where $K:=K_1 \cup K_2$ is lexicographically ordered ($x_1<x_2 \ \ \forall x_i \in K_i$). 
	Clearly, $K$ is a compact linearly ordered space. 
	Define the following action 
	$$
	G \times K \to K, \ \ (g_1,g_2)(x)=
	\begin{cases}g_1(x) \ \ \ x \in K_1 \\
		g_2(x) \ \ \ x \in K_2  
	\end{cases}  
	$$ 
	Then this action is continuous and effective.  Now
	it is enough to check that the induced injective continuous homomorphism $h: G \to H_+(K)$ is a topological embedding. 
	
	Let $U=U_1 \times U_2$ be a \nbd of $e$ in $G$. We claim that there exist finitely many points $\alpha:=\{x_1, \cdots, x_n\}$ in $K$ 
	and their neighborhoods $\gamma:=\{V_1, \cdots,  V_n\}$ in $K$ such that 
	$$\{g \in G: \ g(x_k) \in V_k\} \subset U.$$

	\sk 
	Note that for every linearly ordered compact space $(Y,\leq)$ the pointwise and compact-open topologies coincide on $H_+(Y,\leq)$ (see Sorin \cite{Sorin}).  
	Therefore, there exist $\alpha_1:=\{a_1, \cdots,a_s\} \subset K_1$ with nbds $\g_1:=\{W_1, \cdots, W_s\}$ in $K_1$ and  $\a_2:=\{b_1, \cdots, b_m\} \subset K_2$ with neighborhoods  
	$\g_2:=\{P_1, \cdots, P_m\}$ in $K_2$ 
	such that 
	$$
	\{g \in G_1: \ g(a_i) \in W_i\} \subset U_1, \ \ \ \{g \in G_2: \ g(b_i) \in P_i\} \subset U_2. 
	$$
	Now define $\a:=\a_1 \cup \a_2$ and $\g:=\g_1 \cup \g_2$.  
\end{proof}

We do not know if the c-orderly analog of Proposition \ref{p:prod} is also true. 

\sk 
\begin{example} \label{e:R^n} 
$\R^n$ is orderly for every $n \in \N$.  Indeed, $\R$ is a topological subgroup of $H_+[0,1]$ (treat $[0,1]$ as the 2-point compactification of $\R$). Now, apply Proposition \ref{p:prod}.
\end{example}

The following question was proposed by the referee and also by V. Pestov. 

\begin{question} \label{q:LC} 
	Which locally compact group is: orderly ? c-orderly ?
\end{question}

At least for orderly case, ``locally compact group" can be replaced by ``Lie group" (Corollary \ref{c:Lie}). Note also that every locally compact subgroup of $H_+(S^1)$ is a Lie group (see Ghys \cite[Theorem 4.7]{Ghys}).

\sk 
\section{Ordered enveloping semigroup compactifications} \label{s:env} 


A semigroup $S$ with a topology $\tau$ is said to be \textit{right topological} if every right translation 
$\rho_s\colon S \to S, x \mapsto xs$ is continuous. As in \cite{BJM},   
a \textit{right topological semigroup compactification} (in short: \textit{rts-compactification}) of a group $G$ is a pair $(\g, S)$ such that $S$ 
is a compact right topological semigroup, and $\g$ is a continuous semigroup homomorphism
from $G$ into $S$, where $\g(G)$ is dense in $S$ and the 
natural action $G \times S \to S$ is continuous. 

The most important example comes from topological dynamics. For every compact $G$-system $K$ the pointwise closure in $K^K$ of the set 
of all $g$-translations $K \to K$ (where $g \in G$) is a compact right topological semigroup which is said to be an (Ellis) \textit{enveloping semigroup} of $K$; notation $E(K)$. See \cite{Gl-env} for more details. 
We get a compactification map $j \colon G \to E(K)$, the so-called \textit{Ellis compactification}. This map is an injection if the action is effective but $j$ is not in general \textit{proper} even if $G$ is embedded into the homeomorphism group $H(K)$ with respect to the compact open topology. 

\sk 
In our opinion, the following topological version of orderability deserves special interest. 
It was inspired by Hindmann and Kopperman \cite{HK} (who dealt with discrete $G$).

\begin{defin} \label{d:StrongOrderly} 
	Let $S$ be a compact right topological (in short: crt) semigroup. We say that $S$ is a \textit{linearly ordered crt-semigroup} if there exists a bi-invariant linear order on $S$ such that the interval topology is just the given topology. 
\begin{enumerate}
	\item A crt-semigroup compactification $\g \colon G \hookrightarrow S$ of a topological group $G$ with a bi-invariant order is an \textit{ordered semigroup compactification} if $S$ is a linearly ordered crt-semigroup such that $\g$ is an order compactification. 
	\item  We say that $G$ is \textit{dynamically orderable} if it admits a \textit{proper} order semigroup compactification  (i.e., $\g \colon G \hookrightarrow S$ is a topological embedding and order embedding). 
\end{enumerate}
	 
\end{defin}

\sk 
Every dynamically orderable group $G$ is orderly as a topological group 
($G$ is embedded into $H_+(S)$) 
and orderable as an abstract group (because the linear order on $S$ is bi-invariant). 


\begin{example} \label{e:exSLin} \ 
	\ben 
	\item The standard two-point compactification $S=\{-\infty\} \cup \R \cup \{\infty\}$ of the reals $\R$ (with the usual topology) 
	is an ordered rts-compactification. Here, the semigroup operation in $S$ (keeping the additive symbol) requires $s + \infty =\infty$ and 
	$s + (-\infty) = -\infty$ for every $s \in S$. Hence, $S$ is not commutative.  
	\item (N. Hindman and R.D. Kopperman \cite{HK}) For every linearly ordered group $(G,\leq)$ with the discrete topology, there exist proper linearly ordered rts-compactifications. Moreover, between them there exists the greatest compactification $G \hookrightarrow \mu G$, which, in fact, is the Nachbin's compactification of $(G, \tau_{discr}, \leq)$.  For example, if $\Q$ is the naturally ordered group of all rationals with the discrete topology then $\Q \hookrightarrow \mu \Q$ is well defined. The corresponding $\mu \Q$ is a nonmetrizable compact right topological semigroup. For an intuitive picture about the semigroup $\mu \Q$ we need: a) to replace any point $q \in \Q$ by the ordered triple $(q^-,q,q^+)$; b) to add the ordered pair $(r^-,r^+)$ for every irrational $r \in \R\setminus \Q$; c) to add the end-elements $-\infty, + \infty$ 
	(more formally, we consider 
	$$
	\{-\infty, + \infty\} \times \{0\} \times \{-1.0.1\} \times (\R \setminus \Q) \times \{-1.1|\} 
	$$
	ordered lexicographically). 
	\een
\end{example}

\begin{defin} \label{d:StrMon} 
Let $(G,\leq_G)$ be a 
group with a linear order $\leq_G$ and  
$(K,\leq)$ be a linearly ordered compact effective $G$-system. We say that the given action is \textit{strongly monotone} if  
the maps $\tilde{x} \colon G \to X, g \mapsto gx$ are order preserving for every $x \in X$. 
\end{defin}

\begin{thm} \label{t:env} \
	\begin{enumerate}
		\item 	For every strongly monotone action of $(G,\leq_G)$ on $(K,\leq)$ 
		the corresponding Ellis  semigroup $E(K)$ is a linearly ordered semigroup and the Ellis compactification $j \colon G \to E(K)$  
		is an injective linearly ordered semigroup compactification.
%
%
		\item If, in addition, $G$ is separable then $E(K)$ is hereditarily separable and first countable. 
	\end{enumerate}
\end{thm}
\begin{proof} (1)  Denote by $C_+(K,[0,1])$ the set of all order-preserving continuous functions $K \to [0,1]$. 
	The Ellis compactification $j \colon G \to E(K)$ is equivalent to the diagonal compactification 
	$j^* \colon G \to cl(j^*(G)) \subset [0,1]^F$, where 
	$$F=\{m(f,a) \colon G \to [0,1], g \mapsto f(ga) \ \ f \in C_+(K,[0,1]), a \in K\}.$$
	In order to see this, first note that 
	every $m(f,a)$ can be extended to a continuous function $E(K) \to [0,1]$. Now it is enough to show that such extensions separate the points of $E(K)$. For every distinct $p,q \in E(K)$ there exists $a \in K$ such that $pa \neq qa$. In order to complete this part of the proof, recall that by Nachbin's theorem $C_+(K,[0,1])$ separates the points of $K$, \cite{Nachbin}. 
	
	
	The desired linear order $\preceq$ on $E(K)$ is defined as follows:
	$$
	s_1 \preceq s_2 \ \text{iff} \ s_1a \leq s_2a \ \ \forall a \in K. 
	$$ 
	It is equivalent to
	$$f(s_1a) \leq f(s_2a) \ \ \forall f \in C_+(K,[0,1]), a \in K$$ (because $K$ is linearly ordered and $C_+(K,[0,1])$ separates the points). Clearly, we get a partial order on the set $E(K)$. 
	Moreover, this partial order extends the linear order on $G=j(G)$
	(use that every orbit map $\tilde{a} \colon G \to X, g \mapsto ga$ is order preserving.). 
	In Claim, 5 we show that this partial order $\preceq$ on $E(K)$ is linear. 
	
	\sk
	Claim 1: $s_1p \preceq s_2p$ for every $s_1 \preceq s_2$ and $p$ in $E(K)$. 
	\sk 
	\begin{proof} 
		$s_1 \preceq s_2$ means that $s_1a \leq s_2a$ for every $a \in K$. Then 
		$$(s_1p)a=s_1(pa) \leq s_2(pa) = (s_2p)a$$ for every $p \in E(K)$. 	
	\end{proof}
	
	\sk 
	Claim 2: This partial order $\preceq$ is closed on $E(K)$. 
	\sk
	\begin{proof}  
		Let $\lim p_i = p, \lim q_i = q$ are convergent nets and $p_i \preceq q_i$ in $E(K)$. Then $p_ia \leq q_ia $ for all $a \in K$.  At the same time,  by definition of the pointwise topology on $E$, we have $\lim p_ia =pa, \lim q_ia=qa$. Since the linear order is closed in $K$ (Lemma \ref{l:LinOrdClosed}.1), we get $pa \leq qa$.   
	\end{proof}

	\sk
	Claim 3: The semigroup $E(K)$ is (left and right) ordered.  
	\sk
	\begin{proof} 
		After Claim 1, now it is enough to show that $ps_1 \preceq ps_2$ for every $s_1 \preceq s_2$ and $p \in E(K)$. 
		Since $j(G)$ is dense in $E(K)$, for a given $p \in E(K)$, there exists a net $g_i \in G$ such that $\lim g_i =p$. Then $\lim g_is_1 = ps_1, \lim g_is_2 = ps_2$ because $E$ is right topological. Since $s_1 \preceq s_2$ and the action of $G$ on $K$ is by order-preserving transformations, we have $g_is_1a \leq g_is_2a$ for every $a \in K$ and every $g_i$.
		Therefore, $g_is_1 \preceq g_is_2$. Now, by Claim 2 we have $ps_1 \leq ps_2$, as desired. 	
	\end{proof}

	\sk 
	Claim 4: $j \colon G \to E(K)$ is an order embedding.   
	\sk 
	\begin{proof} Let $g_1, g_2$ be distinct elements in $G$. 
		If $g_1 \leq_G g_2$ then 
		$g_1x \leq g_2x$ for every $x \in K$. 
		Hence, $j(g_1) \preceq j(g_2)$. Conversely, let $g_1x \leq g_2x$ for every $x \in K$. Since the action of $G$ on $K$ is effective, there exists $x_0 \in K$ such that $g_1x_0 \neq g_2x_0$. So, $g_1x_0 < g_2x_0$. Then,  necessarily, $g_1 < g_2$
		 (otherwise, the orbit map $\tilde{x_0}: G \to K$ is not order preserving). 
	\end{proof}

	\sk
	
	Claim 5: $\preceq$ is a linear order on $E(K)$. 
	
	\sk 
	\begin{proof}
Since the partial order $\preceq$ is closed in $E(K)$ (Claim 2) and $j(G)$ is a dense linearly ordered subset  (Claim 4), we can apply Lemma \ref{l:DenseLin}. 
	\end{proof} 


(2) Use Remark \ref{r:Ost} below. 
\end{proof}


\begin{remark} \label{r:inj+} 
		If in Theorem \ref{t:env} $G$ carries the pointwise topology with respect to the action on $K$, then $j$ is a topological embedding. 
\end{remark}

\begin{remark} \label{r:Ost} 
	For separable (e.g., countable) $G$, the space $E(K)$  from Theorem \ref{t:env} is separable. 
	Moreover, then $E(K)$ is hereditarily separable and first countable. 
	It is important to take into account that, in general, every compact linearly ordered separable space $S$ comes from a double arrow type construction. Namely, $S$ is homeomorphic to a special linearly ordered space $X_A$ which can be obtained using a splitting points construction. By a result of Ostaszewski (see \cite{Ost} and its reformulation \cite[Result 1.1]{Mar}) for $S$ there exist: a closed subset $X \subset [0,1]$ and a subset $A \subset K$ such that $X_A=(X \times \{0\} \cup (A \times \{1\}))$ is endowed with the corresponding lexicographic order inherited from $X \times \{0,1\}$. Note that such $X_A$ always is hereditarily separable and first countable. $X_A$ is  metrizable if and only if $A$ is countable. It would be interesting to know when $E(K)$ is metrizable. (cf. Question \ref{q:MetrizSemComp}).  
\end{remark}

\begin{question}  \label{q:c-version} 
	Does the circular version of Theorem \ref{t:env} remains true ? 
\end{question}

\begin{thm} \label{t:bi-Case}  
	Let $G$ be an abstract group. 
	\ben 
	\item  
	The following are equivalent:
	\begin{itemize}
		\item [(a)] $G$ is orderable;  
		\item [(b)] $G$ is dynamically orderable 
		
		\noindent (i.e., $G$ admits an order semigroup compactification $\g \colon G \to S$ which is a topological embedding of the discrete group $G$). 
	\end{itemize}
	\item 
	If $G$ is countable then one may choose $S$ such that, in addition, $S$ is first countable and hereditarily separable.  
	\een 
\end{thm} 
\begin{proof} (1) 
	(b) $\Rightarrow$ (a) The order inherited from the ordered semigroup $S$ on its subgroup $G=j(G)$ is bi-invariant.

	(a) $\Rightarrow$ (b) 
	Take a bi-invariant order $\leq$ on $G$ and consider the order $G$-compactification $X \to X_{\infty}$, where $X=(G,\leq)$; see Theorem  \ref{t:limLin}. Now 
	apply Theorem \ref{t:env} to the compact $G$-system $K:=X_{\infty}$. Observe that in this case $j \colon G \to S=E(K)$ necessarily is a topological embedding because discrete $X:=G$ is embedded into $K$ and the orbit map  $$\tilde{a} \colon  E(K) \to K, \ p \mapsto pa$$ is continuous for the unit element $a:=e \in X \subset K$. 
	
	(2) Apply Theorem \ref{t:env}.2 or Remark \ref{r:Ost}. 
%
\end{proof}

The first assertion (1) of Theorem \ref{t:bi-Case} can be derived also from \cite[Theorem 2.1]{HK}. 

\sk 
Taking into account Remark \ref{r:Ost}, the following question seems to be attractive. 

\begin{question} \label{q:MetrizSemComp}  
	Which countable ordered 
	discrete groups $G$
	admit a proper order crt-semigroup compactification $G \hookrightarrow S$ with metrizable $S$.  
	What about the free group $F_2$? 
\end{question}

\sk  \sk 
\section{Circularly ordered groups}
\sk

The notion of cyclically ordered group is due to L. Rieger \cite{Reiger}. This concept was studied in several directions. 
See, for example, \cite{Zheleva97,Calegari04,BS,CMR}).

 Recall that 
a c-order-preserving action of $G$ on a circularly ordered set $(X,\circ)$ is defined as follows 
$$
[x,y,z] \Leftrightarrow [gx,gy,gz] \ \ \forall \ g \in G, \ x,y,z \in X. 
$$ 
That is, if  $g$-translations $X \to X, x \mapsto gx$ are COP maps for every $g \in G$. 
In particular, one may define \textit{left circularly orderable} (in short: \textit{left c-orderable}) and \textit{circularly orderable} (in short: \textit{c-orderable}) groups. Precisely, this means that $G$ admits a circular order $\circ$ such that the left action (resp., left and right actions) of $G$ on itself is  c-order-preserving.  
 Abbreviation: 
L-COrd and COrd, respectively. 

Every L-Ord group is L-COrd (see Remark \ref{r:chech}.1) and every Ord group is COrd. A finite group is L-COrd iff it is a cyclic group iff it is COrd. The circle group $\T$ is COrd but not Ord. A countable group G is L-COrd iff it acts faithfully on $\T$ by orientation (circular order) preserving homeomorphisms (see, for example,  \cite{Calegari04,BS}). That is, iff $G$ algebraically can be embedded into $H_+(\T)$.


\sk  
%

Note that the world of circular orders is much larger than the world of linear orders. 
For example, $\Z$ admits only two linear orders. However, $\Z$ has continuum many circular orders. 
This contrast is very sharp also in dynamical systems.      
One of the most dramatic differences is for compact minimal $G$-systems. Every minimal linearly ordered $G$-system is trivial (Lemma \ref{l:min}). In contrast, we have many nontrivial circularly ordered $G$-systems in symbolic dynamics. See \cite{GM-c} for further discussion and also Remark \ref{r:M(G)}.

\sk 
  A representation theorem for cyclically ordered groups was proved by S. Swierczkowski \cite{Scw}. This theorem asserts that every cyclically ordered group can be embedded into a lexicographic product $\T \otimes L$ of the circle group $\T$ and an appropriate linearly ordered group~$L$.  


In Definition \ref{d:orderly} c-orderly (meaning, circular orderly) topological groups were defined  as topological subgroups of $H_+(K)$ for some compact circularly ordered space $(K,\circ)$. 
		For every circularly ordered set $(X,\circ)$ the topological group $\Aut(X,\leq)$ with its pointwise topology with respect to the discrete set $X$ is c-orderly (Corollary \ref{c:withPointw}).

\begin{prop}
	Every orderly group is c-orderly. 
\end{prop}
\begin{proof}
	Use Lemma \ref{l:closed}. 
	\end{proof}

The equivalence of (1) and (3) in Theorem \ref{t:c-ordCase} is known (Zheleva \cite{Zheleva97}). 
%
We present below a direct proof of Theorem \ref{t:c-ordCase} which 
strengthens results of Zheleva and probably has intrinsic interest. 

\begin{thm} \label{t:c-ordCase} 
	Let $G$ be an abstract group. 
	The following are equivalent:
	\ben 
	\item $G$ is left-c-orderable;   
	\item $(G,\tau_{discr})$ is c-orderly \nl (i.e., a discrete copy of $G$ topologically is embedded into the topological group $H_+(K,\circ)$ for some compact circularly ordered space $K$);  
	\item $G$ algebraically is embedded into the group $\Aut(X,\circ)$ for some circularly ordered set $(X,\circ)$.  
	\een
\end{thm} 
In the assertion (2), in addition, we can suppose that $\dim K=0$.
\begin{proof}
	(1) $\Rightarrow$ (2)  
	Let $(G,R)$ be a c-ordered group. By Theorem \ref{t:lim}, the c-compactification $\nu \colon G \to K=G_{\infty}$ is a c-order-preserving proper $G$-compactification which induces a topological embedding of (discrete) $G$ into $H_+(K)$.  
	
(2) $\Rightarrow$ (3) Trivial. 

(3) $\Rightarrow$ (1) Apply Corollary \ref{t:S_+} below.  
\end{proof}

The idea of the following definition comes from c-ordered lexicographic products. In order to compare with Definition \ref{d:lexic}), consider the projection  $C \otimes L \to C$ as a particular case of Definition \ref{d:MainDef}.  This construction can be found in \cite[Def. II.27, p. 72]{DK15}. 

\begin{defin} \label{d:MainDef}  
Let $(Y,R_Y)$ be a circularly ordered set and $q \colon X \to Y$ be an onto map such that for every fiber $X_y:=q^{-1}(y)$ we have a linear order $L_y=\leq_y$. 
Then  we have a canonical circular
order $R_X$ on $X$ (as we show below) defined as follows:  
$[x_1,x_2,x_3]$ if one of the following conditions is satisfied:
\ben 
\item $[q(x_1),q(x_2),q(x_3)]$; 
\item $q(x_1)=q(x_2) \neq q(x_3)$, \ $x_1 < x_2$;
\item $q(x_2)=q(x_3) \neq q(x_1)$, \  $x_2 < x_3$;
\item $q(x_1)=q(x_3) \neq q(x_2)$, \ $x_3 < x_1$;

\item $q(x_1)=q(x_2)=q(x_3), \ [x_1,x_2,x_3]_{L_y}$. 
\een
\end{defin}

\begin{lem} \label{l:mainTechn} 
	For $R_X$ defined in Definition \ref{d:MainDef}  we have:  
\bit
\item [(a)] 
 $R_X$ is a circular order on $X$.  
\item [(b)] $R_X$ is compatible with $L_y$ on every $X_y$, $y \in Y$. 
\item [(c)] $q \colon X \to Y$ is c-order preserving.  
\item [(d)] 
Let $G$ act on $X$ and $Y$ such that $q \colon X \to Y$ is a $G$-map, 
$G \times Y \to Y$ is c-order preserving and $gx_1 < gx_2$ for every $x_1 < x_2$ in $X_y$, every $y \in Y$ and $g \in G$. Then the action $G \times X \to X$ is c-order preserving. 
\eit 
\end{lem}
\begin{proof}
	(a) The Cyclicity, Asymmetry and Totality axioms are trivial to verify.  
	 Transitivity is straightforward though nontrivial because there are 
	 many cases to check. 
	\sk 
	(b) Directly follows from the definitions. 
	\sk 
	
	(c) Let $[x_1,x_2,x_3]$ and $q(x_1),q(x_2),q(x_3)$ be distinct. Then $[q(x_1),q(x_2),q(x_3)]$ by Definition \ref{d:MainDef}.  
	
	If $q(x_1)=q(x_2)$, then $x_1,x_2$ both are elements of the linearly ordered subset $X_y$, where $y=q(x_1)$. Let $x_1 < x_2$. Then the interval $[x_1,x_2]_R$ is just the usual interval $\{x \in X: x_1 < x < x_2\}$ in $X_y$. Then $q$ is constant on this interval, $q([x_1,x_2]_R))=y$. So, $q$ is c-order preserving (by Definition \ref{d:c-ordMaps}).  
	
	\sk 
	(d) It is straightforward. 
	\end{proof}

%

\sk 

\begin{cor} \label{t:S_+} \emph{(Zheleva \cite{Zheleva97})}
	Let $(X,\circ)$ be a c-ordered set and $G \times X \to X$ is an effective c-order preserving action. Then the group $G$ (e.g., $\Aut(X,\circ)$) admits a left invariant c-order. 
\end{cor}
\begin{proof}
	Take any $a \in X$ and define the linearly ordered set $(X,\leq_a)$ (as in Remark \ref{r:chech}). Let $H:=St(a) \subset G$ be the stabilizer subgroup of $a$. 
	Then we have the restricted action $H \times X \to X$. Which, clearly, is also \textit{effective}. This action preserves the order $\leq_a$. Indeed, if $[a,x,y]$ then $[ga,gx,gy]$ for every $g \in G$. In addition for every $h \in H$ we have 
	$ha=a$. Therefore, $[a,hx,hy]$. This means that $hx <_a hy$ for every $x<_ay$. 
	By Theorem \ref{t:LinCase}, $H$ admits a natural left invariant linear order $\prec_H$. 
	
	Now consider the orbit $G$-map
	$$q \colon G \to Ga=G/H, \ g \mapsto ga.$$ 
	The action of $G$ on $Ga$ is c-order preserving and it is $G$-equivalent to the usual left action of $G$ on $G/H$. 
		On every coset $gH$ define the induced linear order 
	$$
	gh_1 \prec gh_2 \equiv h_1 \prec_H h_2. 
	$$ 
	Since the linear order on $H$ is left invariant, it is easy to see that 
	this binary relation is well defined and does not depend on the choice of $a \in G$ with $gH=aH$.  Clearly, it defines a linear order on $gH$ which coincides with $\prec_H$ on $H$. Note that $q^{-1}(ga)=	gH$. 
	Now we can apply Lemma \ref{l:mainTechn} to the onto orbit map 
	$q \colon G \to Ga$. 
\end{proof}

The proof of Corollary \ref{t:S_+} can be easily adopted in order to prove 
the following sufficient condition of c-orderability. 

\begin{cor} \label{t:H} \emph{(H. Baik and E. Samperton \cite[Lemma 4.4]{BS})} 
	Let $X$ be a set which admits a $G$-invariant circular order with respect to a given action $G \times X \to X$  of a group $G$. Suppose that the stabilizer subgroup $St(a)$ of a point $a \in X$ is left linearly ordered. Then $G$ has a left invariant circular order such that $G \to X, g \mapsto ga$ is c-order preserving.  
	%
	%
\end{cor}
%

This result, in turn, implies another sufficient condition (in terms of short exact sequences)  due to D. Calegari \cite[Lemma 2.2.12]{Calegari04}.

\sk 
\subsection*{Circularly ordered enveloping semigroup compactifications} \label{s:c-env} 

\begin{defin} \label {d:cStrongOrderly} 
	Let $S$ be a 
	crt semigroup. We say that $S$ is a \textit{c-ordered crt-semigroup} if there exists a bi-invariant  c-order on $S$ such that the circular topology is just the given topology. A crt-semigroup compactification $\g \colon G \to S$ of a topological group $G$ with a bi-invariant c-order is an \textit{c-ordered semigroup compactification} if $S$ is a c-ordered crt-semigroup such that $\g$ is a c-order compactification.  

We say that $G$ is \textit{dynamically c-orderable} if it admits a \textit{proper} c-order semigroup compactification  (i.e., $\g$ is a topological embedding and c-order embedding).  
\end{defin}

\begin{example} \label{e:exS} \ 
	\ben 
	\item \cite[Cor. 6.5]{GM-c} (Sturmian-like systems) 
	
	\nt For every irrational $\alpha \in \R$ and every $R_\a$-invariant subset 
	$A \subset \T$, one may define a circularly ordered cascade ($\Z$-system) $\T_A$ which we get from the circle $\T$ by replacing any point $a \in A$ by the ordered pair $(a^{-}, a^{+})$.  
	In this case, the enveloping semigroup of $\T_A$ 
	is the circularly ordered cascade $\T_{\T} \cup \Z$. It contains the double-circle $\T_{\T}$ 
	(\emph{c-ordered lexicographic product} $\T \times \{-,+\}$) as its unique minimal left ideal; or equivalently as its unique minimal subset.

	\item \cite[Example 14.10]{GM1} Consider in more detail a special case of (1),  
	when $A$ is a subgroup of $\Z$, generated by an irrational $\al \in \T$.  
	For convenience we will use the notation $\beta^{\pm}$ ($\beta \in \T$) for points of $K:=\T_A$, where $\beta^-=\beta^+$ for every $\beta \in \T \setminus A$. We have a  homeomorphism 
	$$\sigma \colon \T_A \to \T_A, \ \sigma (\beta^{\pm})=(\beta + \al)^{\pm}$$
	which defines a circularly ordered $\Z$-system $\T_A$.

	Then the corresponding enveloping semigroup 
	$E= E(\T_A)$ can be identified with the 
	disjoint union 
	$$ \T_{\T} \cup \{\sigma^n : n \in \Z\},$$ where $(\T_{\T},\sigma)$ is the Ellis' {\em double circle} cascade:
	$\T_{\T} = \{\beta^{\pm} : \beta \in \T = [0,1)\}$ and $\sigma \circ \beta^{\pm} = (\beta + \al)^{\pm}$. 
	One may show that 
	$E$ becomes a circularly ordered semigroup, where $E=\T_{\T} \cup \Z$ is a c-ordered subset of 
	the c-ordered lexicographic order $\T \times \{-,0,+\}$. 
	Under this definition for every $n \in \Z$ we have	$[n\al^-, \sigma^n, n\al^+]$. 
	Since the interval $(n \alpha^+, n \alpha^-) \subset E$ contains only the single element $\sigma^{n}$ for every  $n \alpha \in G=\Z$, we get that every element of $G=j(G)$ is isolated in $E$. So, in this case $E =\T_{\T} \cup \Z$, where each point of $\Z$ is isolated in $E$. 
	
	For every $\g \in \T$, define the following self-maps: 
	$$
	p_{\g}^+ \colon \T_A \to \T_A, \ p_{\g}^+(\beta^{\pm})=(\beta+\g)^+,
	$$
	$$ 
	\ p_{\g}^- \colon \T_A \to \T_A, \ p_{\g}^-(\beta^{\pm})=(\beta+\g)^-. 
	$$
	Then $E(\T_A) = \{p_{\g}^{\pm} \colon \g \in \T\} \cup \Z.$ 
	It is straightforward to show that $E(\T_A)$ is an ordered semigroup. In particular, observe that the left and right translations on $E(\T_A)$ are circular order-preserving maps (in the sense of Definition \ref{d:c-ordMaps}). 
	\een
\end{example}

\begin{thm} \label{t:c-Case} 
	Let $(G,\tau)$ be an abstract discrete group. 
	The following are equivalent:
	\ben 
	\item $G$ is circularly orderable;  
	\item $G$ is dynamically c-orderable 
	
\nt	(i.e., $G$ admits a c-order proper semigroup compactification $\g \colon G \hookrightarrow S$).  
	
	\een
\end{thm}
\begin{proof}
	(1) $\Rightarrow$ (2) Let $R$ be a bi-invariant circular order on $G$. By a classical result of Scwierczkowski \cite{Scw}, there exists a group embedding $i\colon G \hookrightarrow \T \otimes_c L$, into the lexicographic c-product of groups, where $\T$ is the usual c-ordered circle group and $L$ is an ordered group. By Theorem \ref{t:bi-Case}, there exists a linear order semigroup compactification $\g\colon L \hookrightarrow P$. Consider the lexicographic product on the semigroup $\T \otimes_c P$. 
	Now observe:
	
	\ben 
	\item [(a)] $\T \otimes_c P$ is compact in the circular topology of the lexicographic c-order by Proposition \ref{p:LexComp}.

	\item [(b)]  For every $u=(t,p) \in \T \otimes_c P$, the left and right translations $\lambda_{u} \colon \T \otimes_c P \to \T \otimes_c P$ and $\rho_{u} \colon \T \otimes_c P \to \T \otimes_c P$ are c-order preserving. 
	The verification is straightforward using Definition \ref{d:lexic}. 
		\item [(c)] For every $g \in G$, the left translations $\lambda_{i(g)}$ (and its inverse) $\lambda_{i(g^{-1})}$ are both circular order-preserving self-maps of $\T \otimes_c P$. Hence, such translations are homeomorphisms because $\T \otimes_c P$ carries the circular topology. Therefore, $i(G)$ is a subset  
		of the topological centre of $\T \otimes_c P$. 
		
		\item [(d)]  $\T \otimes_c P$ is a right topological semigroup. 
		Again, we use Definition \ref{d:lexic} in order to show that the preimage of every interval under any given right translation $\rho_{(t_0,p_0)}$ is open. 
	\een 
	Finally, define $S$ as the closure of $i(G)$ in $\T \otimes_c P$. Then $S$ is a compact right topological semigroup, and the subspace (compact) topology coincides with the circular topology under the inherited circular order (indeed, the Hausdorff circular topology is weaker than the compact subspace topology). Moreover, the induced embedding $\g \colon G \hookrightarrow S=cl(i(G))$ is a c-order semigroup compactification. 
	\sk 
	
	(2) $\Rightarrow$ (1) is trivial. 
\end{proof}
 



\sk 
Every (c-)ordered compact $G$-space $K$ is \textit{tame} in the sense of A. K\"{o}hler \cite{Ko} \textit{(regular}, in the original terminology).  
If $K$ is metrizable then it is equivalent to say that the enveloping semigroup $E(K)$ is a separable Rosenthal compact space (see \cite{GM-survey,GM-c}). 

\sk 
Deep results of Todor\u{c}evi\'{c} \cite{T} and Argyros--Dodos--Kanellopoulos  \cite[Section 4.3.7]{A-D-K})
about separable Rosenthal compacta, lead to a hierarchy of tame metric dynamical systems (see \cite{GM-TC}) according to topological properties of corresponding enveloping semigroups. In view of this hierarchy we ask the following

\begin{question} \label{q:SmallEnv2} 
	Which (c-)orderly topological groups $G$ admit an effective (c-)ordered continuous action 
	on a compact metrizable space $K$ such that the Ellis compactification $G \hookrightarrow E(K)$ is a topological embedding and the enveloping semigroup $E(K)$ 
	is: 
	a) metrizable? b) hereditarily separable? c) first countable? 
\end{question}

\sk \sk 
\section{Representations on Banach spaces} \label{s:repres} 

Banach representations of dynamical systems probably can provide an interesting direction for estimating the complexity of orderable groups.  
Recall that according to results of \cite{GM-c}, 
every (c-)ordered compact $G$-space $K$ is representable on a \textit{Rosenthal Banach space} (not containing an isomorphic copy of $l^1$). It follows that every c-orderly topological group $G$ (e.g., every $H_+(K)$) is Rosenthal representable. That is, $G$ is embedded as a topological subgroup into the linear isometry group $\Iso_l(V)$ (with the strong operator topology) for some Rosenthal Banach space $V$. However, it is not always true if $V$ is an Asplund space. 
In contrast, the actions of topological groups $H_+[0,1]$ and $H_+(S^1)$ on $[0,1]$ and $S^1$, respectively, 
are not Asplund representable (Theorem \ref{t:notAspOrd}).   

 Recall that a Banach space $V$ is said to be \textit{Asplund} if the dual $W^*$ of every separable Banach subspace $W \subset V$  is separable. Every Asplund space is Rosenthal.

\begin{defin} \label{d:AspOrd} 
	Let us say that a topological group $G$ is \textit{Asplund orderly} if 
	$G$ topologically can be embedded into $H_+(K)$, where $K$ is a linearly ordered compact space such that the $G$-space $K$ is Asplund representable. Similarly can be defined \textit{Asplund c-orderly} groups. 
\end{defin}

\begin{question} \label{q:AspOrd}
	Which (c-)orderly topological groups are 
	Asplund (c-)orderly? 
\end{question}

\begin{remark} \label{t:notAspOrd} \ 
	\begin{enumerate}  
		
		\item If, in Definition \ref{d:AspOrd}, in addition, $K$ is metrizable, then $K$, as a $G$-space, is Asplund representable iff the enveloping semigroup $E(K)$ is metrizable (this follows from \cite{GMU08}).  This gives an important link between Questions \ref{q:SmallEnv2} and \ref{q:AspOrd}. 
			\item 	The orderly topological group $H_+[0,1]$ and the c-orderly topological group $H_+(S^1)$ are not Asplund orderly and Asplund c-orderly, respectively. Indeed, according to 	\cite{GM-NewAlg08, GM-tLN}, any representation of the groups $H_+[0,1]$ and $H_+(S^1)$ on an Asplund Banach space is trivial.
				\item However by \cite{GM-c}, 
				the actions of the Polish groups $H_+[0,1]$ and $H_+(S^1)$ on $[0,1]$ and $S^1$, respectively, admit a proper representation on a Rosenthal 
				Banach space. 
	\end{enumerate} 
\end{remark}




%

\begin{remark} \label{r:RnAspOrd} 
	The standard two-point compactification of $\R$ is an ordered rts-compactification which is an Asplund representable $\R$-system  (see \cite{GM1}) but not reflexive representable. Hence, $\R$ is Asplund orderly. Proposition \ref{p:ProdAsp-Ord} implies, in particular, that $\R^n$ is Asplund-orderly. 
\end{remark}

\begin{prop} \label{p:ProdAsp-Ord}  
	Finite product of Asplund-orderly topological groups is Asplund-orderly. 
\end{prop}
\begin{proof}
	This can be done similarly to the proof of Proposition \ref{p:prod} taking into account that if a compact $G_i$-space $K_i$ ($i \in \{1,2\}$) is Asplund representable,  then 
the $G$-space $K$ is also Asplund representable, where $G:=G_1 \times G_2$ and $K$ is the disjoint sum $K_1 \cup K_2$. In order to check this, we use a criterion (see \cite[Theorem 3.11.2]{evFr} or \cite{Me-b}) of Asplund representability. This criterion asserts that a compact $G$-space $X$ is Asplund representable iff there exists a $G$-invariant point-separating bounded family $F$ of real functions on $X$ which is a \textit{fragmented family} (in the sense of \cite{GM1}) on $X$. 
	Using this result let $F_i$ be a such a family for the $G_i$-space $K_i$.  
	For every $f\colon K_1 \to \R$ from $F_1$ define $\overline{f_1} \colon K \to \R$, where $K:=K_1 \cup K_2$ with $\overline{f_1}(y)=0$ for every $y \in K_2$. 
	Set $\overline{F_1}:=\{\overline{f} \colon f \in F_1\}$. Similarly can be defined $\overline{F_2}$. Finally, observe that the family $F:=\overline{F_1}\cup \overline{F_2}$ is a $G$-invariant, bounded, fragmented family of functions on $K$ which separates the points of $K$.  
	For simplicity we omit other details.
	\end{proof} 

\begin{question} \label{q:CountAspOrd} 
	Is it true that every countable discrete (c-)ordered group is
	 Asplund (c-) orderly?  
\end{question}

As we know (Theorem \ref{t:bi-Case}.2) for every discrete countable orderly group $G$ there exists 
an order preserving action of $G$ on a compact metrizable linearly ordered space $K$ which induces a topological embedding $(G,\tau_{discr}) \hookrightarrow H_+(K)$. In order, to ensure that such $G$-space $K$ is Asplund representable, it is enough to show (by \cite{GMU08}) that the enveloping semigroup $E(K)$ is metrizable 
(cf. Questions \ref{q:SmallEnv2}(a) and \ref{q:MetrizSemComp}).

\sk \sk 


\bibliographystyle{amsplain}

\end{document}